\documentclass[a4paper,10pt]{amsart}

\usepackage[utf8]{inputenc} % Linux
\usepackage{amsmath}
\usepackage{amsbsy}
\usepackage{amssymb}
\usepackage{amscd,amsthm}
\usepackage{verbatim}
\usepackage[english]{babel}
\usepackage{dsfont}
\usepackage{anysize}
\usepackage{hyperref}

\newcommand{\Z}{\mathds{Z}}

\newcommand{\N}{\mathds{N}}
\newcommand{\R}{\mathds{R}}

\newcommand{\p}{\phantom}
\newcommand{\q}{\quad}

\newtheorem{thm}{Theorem}[section]
\newtheorem{lem}[thm]{Lemma}
\newtheorem{kor}[thm]{Corollary}
\newtheorem{conj}[thm]{Conjecture}
\newtheorem{prop}[thm]{Proposition}

\theoremstyle{definition}
\newtheorem*{defi}{Definition}

\theoremstyle{remark}
\newtheorem*{rema}{Remark}

\title{On the arithmetic and geometric means of the prime numbers}
\author{Christian Axler}
\address{Institute of Mathematics\\ Heinrich-Heine University Düsseldorf\\
40225 Düsseldorf, Germany}
\email{christian.axler@hhu.de}
\date{\today}
\subjclass[2010]{Primary 11N05; Secondary 11A41}
\keywords{arithmetic mean, geometric mean, prime number, Chebyshev's $\vartheta$-function}

\begin{document}

\begin{abstract}
In this paper we establish explicit upper and lower bounds for the ratio of the arithmetic and geometric means of the prime numbers, which improve the current
best estimates. Further, we prove several conjectures related to this ratio stated by Hassani. In order to do this, we use explicit estimates for the prime 
counting function, Chebyshev's $\vartheta$-function and the sum of the first $n$ prime numbers.
\end{abstract}

\maketitle

%\tableofcontents

\section{Introduction}

Let $a_n$ be the arithmetic mean and $g_n$ be the geometric mean of the first $n$ positive integers, respectively. Stirling's approximation for $n!$ implies
that $a_n/g_n \to e/2$ for $n \to \infty$. In his paper \cite{hassani2013}, Hassani studied the arithmetic and geometric means of the prime numbers, i.e.
\begin{displaymath}
A_n = \frac{1}{n}\sum_{k \leq n}p_k, \q G_n = ( p_1 \cdot \ldots \cdot p_n)^{1/n}.
\end{displaymath}
Here, as usual, $p_k$ denotes the $k$th prime number. By setting $D(n) = \log p_n - \vartheta(p_n)/n$ and $R(n) = \sum_{k \leq n} p_k/n - p_n/2$, where
Chebyshev's $\vartheta$-function is defined by $\vartheta(x) = \sum_{p \leq x} \log p$, Hassani \cite[p. 1595]{hassani2013} derived the identity
\begin{equation}
\log \frac{A_n}{G_n} = D(n) + \log \left( 1 + \frac{2R(n)}{p_n} \right) - \log 2  \tag{1.1} \label{1.1}
\end{equation}
for the ratio of $A_n$ and $G_n$, which plays an important role in this paper. First, we establish some asymptotic formulae for the quantities $D(n)$, $G_n$ and
$A_n$ which help us to find the following asymptotic formula for the ratio of $A_n$ and $G_n$. Here, let $r_t = (t-1)!(1-1/2^t)$ and the positive integers 
$k_1, \ldots, k_s$, where $s$ is a positive integer, are defined by the recurrence formula $k_s + 1!k_{s-1} + \ldots + (s-1)!k_1 = s \cdot s!$.

\begin{thm}[See Theorem \ref{thm206}] \label{thm101}
For each positive integer $m$, we have
\begin{displaymath}
\frac{A_n}{G_n} = e\left( \frac{1}{2} + \sum_{i = 1}^m \frac{1}{\log^ip_n} \left( -r_{i+1} + r_i + \sum_{s=1}^{i-1} r_s k_{i-s} \right) \right) \cdot \exp
\left( \sum_{j = 1}^m\frac{k_j}{\log^jp_n} \right) + O \left( \frac{1}{\log^{m+1}p_n} \right).
\end{displaymath}
\end{thm}

One of Hassani's results \cite[p. 1602]{hassani2013} is that $A_n/G_n = e/2 + O(1/\log n)$, which implies that the ratio of $A_n$ and $G_n$ also tends to $e/2$
for $n \to \infty$. Setting $m = 2$ in Theorem \ref{thm101}, we get the following more accurate asymptotic formula
\begin{equation}
\frac{A_n}{G_n} = \frac{e}{2} + \frac{e}{4\log p_n} + \frac{e}{\log^2 p_n} + O \left( \frac{1}{\log^3p_n} \right). \tag{1.2} \label{1.2}
\end{equation}
%\begin{kor}[See Corollary \ref{kor209}] \label{kor102}
%We have
%\begin{displaymath}
%\frac{A_n}{G_n} = \frac{e}{2} + \frac{e}{4\log p_n} + \frac{e}{\log^2 p_n} + O \left( \frac{1}{\log^3p_n} \right).
%\end{displaymath}
%\end{kor}
%
%Then we use a result of Cipolla \cite[p. 139]{cp} to derive the following asymptotic expansion for the ratio of $A_n$ and $G_n$.
%
Using explicit estimates for the $n$-th prime number and the prime counting function $\pi(x)$, which denotes the number of primes not exceeding $x$, Hassani
\cite[Theorem 1.1]{hassani2013} found some explicit estimates for the ratio of $A_n$ and $G_n$. The proof of these estimates consists of three steps. First, 
Hassani gave some explicit estimates for the quanities $D(n)$ and $\log(1+2R(n)/p_n)$ and then he used \eqref{1.1}. We follow this method to refine Hassani's 
estimates by showing the following both results in the direction of \eqref{1.2}.

\begin{thm}[See Corollary \ref{kor602}] \label{thm102}
For every positive integer $n \geq 62$, we have
\begin{displaymath}
\frac{A_n}{G_n} > \frac{e}{2} + \frac{e}{4 \log p_n} + \frac{0.61e}{\log^2 p_n}.
\end{displaymath}
\end{thm}

\begin{thm}[See Theorem \ref{thm604}] \label{thm103}
For every positive integer $n \geq 294\,635$, we have
\begin{displaymath}
\frac{A_n}{G_n} < \frac{e}{2} + \frac{e}{4\log p_n} + \frac{1.52e}{\log^2 p_n}.
\end{displaymath}
\end{thm}

Since the computation of $p_n$ is difficult for large $n$, the estimates for the ratio of $A_n$ and $G_n$ obtained in Theorem \ref{thm102} and Theorem
\ref{thm103} are ineffective for large $n$. Hence, we are interested in estimates for $A_n/G_n$ in terms of $n$. For this purpose, we find the following 
estimates.

\begin{thm}[See Theorem \ref{thm601}] \label{thm104}
For every positive integer $n \geq 139$, we have
\begin{displaymath}
\frac{A_n}{G_n} > \frac{e}{2} + \frac{e}{4\log n} - \frac{e(\log \log n - 2.8)}{4 \log^2n}.
\end{displaymath}
\end{thm}

\begin{thm}[See Corollary \ref{kor605}] \label{thm105}
For every positive integer $n \geq 2$, we have
\begin{displaymath}
\frac{A_n}{G_n} < \frac{e}{2} + \frac{e}{4\log n} - \frac{e( \log \log n - 6.44)}{4 \log^2n}.
\end{displaymath}
\end{thm}

In particular, we prove several conjectures concerning $D(n)$, $G_n$ and the ratio of $A_n$ and $G_n$ stated by Hassani \cite{hassani2013} in 2013. For
instance, we use Theorem \ref{thm102} to show that the ratio of $A_n$ and $G_n$ is always greater than $e/2$.

\section{Several asymptotic formulae}

Let $\pi(x)$ denotes the number of primes not exceeding $x$. In this section, we give some asymptotic formulae for the quantities $D(n)$, $G_n$, $A_n$, the 
ratio of $A_n$ and $G_n$ and finally for $\log(1+2R(n)/p_n)$. For this purpose, an asymptotic formula for the prime counting function $\pi(x)$ plays an 
important role.

\subsection{Two asymptotic formulae for $D(n)$}

In order to find the first asymptotic formula for
\begin{displaymath}
D(n) = \log p_n - \frac{\vartheta(p_n)}{n}
\end{displaymath}
in terms of $p_n$, we introduce the following definition.

\begin{defi}
Let $m$ be a positive integer. The positive integers $k_1, \ldots, k_m$ are defined by the recurrence formula
\begin{equation}
k_m + 1!k_{m-1} + 2!k_{m-2} + \ldots + (m-1)!k_1 = m \cdot m!. \tag{2.1} \label{2.1}
\end{equation}
In particular, $k_1 = 1$, $k_2 = 3$, $k_3 = 13$ and $k_4 = 71$.
\end{defi}

Then, we obtain the following result.

\begin{prop} \label{prop201}
Let $r$ be a non-negative integer. Then
\begin{displaymath}
D(n) = 1 + \frac{k_1}{\log p_n} + \frac{k_2}{\log^2p_n} + \ldots +  \frac{k_r}{\log^r p_n} + O \left( \frac{1}{\log^{r+1}p_n}  \right).
\end{displaymath}
\end{prop}

\begin{proof}
The proof of the required asymptotic formula for $D(n)$ consists of two steps. First, we find an asymptotic formula for $\log x$. Using a result shown by
Panaitopol \cite{pan3}, we get
% \begin{equation}
% \pi(x) = \frac{x}{ \log x - 1 - \frac{k_1}{\log x} - \frac{k_2}{\log^2x} - \ldots -  \frac{k_{r+1}}{\log^{r+1} x} } + O \left( \frac{x}{\log^{r+3}x}  \right).
% \tag{2.2} \label{2.2}
% \end{equation}
% Hence,
\begin{equation}
\log x = \frac{x}{\pi(x)} + 1 + \frac{k_1}{\log x} + \frac{k_2}{\log^2x} + \ldots +  \frac{k_r}{\log^rx} + O \left( \frac{x}{\pi(x)\log^{r+2}x}  \right). 
\tag{2.2} \label{2.2}
\end{equation}
%Further, \eqref{2.2} implies that
%\begin{equation}
%\lim_{x \to \infty} \frac{\pi(x)}{x/\log x} = 1, \tag{2.4} \label{2.4}
%\end{equation}
The Prime Number Theorem states that $\pi(x) \sim x/\log x$ for $x \to \infty$. So, we can simplify the error term in \eqref{2.2} to obtain
\begin{equation}
\log x = \frac{x}{\pi(x)} + 1 + \frac{k_1}{\log x} + \frac{k_2}{\log^2x} + \ldots +  \frac{k_r}{\log^rx} + O \left( \frac{1}{\log^{r+1}x}  \right). \tag{2.3} 
\label{2.3}
\end{equation}
Next, we establish an asymptotic formula for $\vartheta(p_n)/n$. A well-known asymptotic formula for Chebyshev's $\vartheta$-function is given by $\vartheta(x) 
= x + O (x\exp(-c \log^{1/10}x) )$, where $c$ is an absolute positive constant (cf. Brüdern \cite[p. 41]{bruedern}). Now, the Prime Number Theorem and the 
fact that $\exp(-c \log^{1/10}x) = O(1/\log^s x)$ for every positive integer $s$ indicate that
%\begin{equation}
%\vartheta(x) = x + O \left( \frac{x}{\log^{r+2}x} \right). \tag{2.6} \label{2.6}
%\end{equation}
%From \eqref{2.4} follows that
%\begin{equation}
%\lim_{n \to \infty} \frac{n}{p_n/\log p_n} = 1 \tag{2.7} \label{2.7}
%\end{equation}
%and combined with \eqref{2.6}, we get
\begin{equation}
\frac{\vartheta(p_n)}{n} = \frac{p_n}{n} + O \left( \frac{1}{\log^{r+1} p_n} \right). \tag{2.4} \label{2.4}
\end{equation}
Combined with \eqref{2.3} and the definition of $D(n)$ we conclude the proof.
\end{proof}

Next, we establish an asymptotic formula for the quantity $D(n)$ in terms of $n$. In order to do this, we first note two useful results of Cipolla \cite{cp}
from 1902 concerning asymptotic formulae for the $n$th prime number $p_n$ and $\log p_n$. Here, $\emph{lc}(P)$ denotes the leading coefficient of a polynomial 
$P \in \R[x]$. 

\begin{lem}[Cipolla, \cite{cp}] \label{lem202}
Let $m$ be a positive integer. Then there exist uniquely determined polynomials $Q_1, \ldots, Q_m \in \Z[x]$ with $\emph{deg}(Q_k) = k$ and $lc(Q_k) = (k-1)!$,
so that
\begin{displaymath}
p_n = n \left( \log n + \log \log n - 1 + \sum_{k=1}^m \frac{(-1)^{k+1}Q_k(\log \log n)}{k!\log^kn} \right) + O\left( \frac{n (\log \log n)^{m+1}}{\log^{m+1}n}
\right).
\end{displaymath}
The polynomials $Q_k$ can be computed explicitly. In particular, $Q_1(x) = x - 2$, $Q_2(x) = x^2 - 6x + 11$ and $Q_3(x) = 2x^3 - 21x^2 + 84x - 131$.
\end{lem}

\begin{lem}[Cipolla, \cite{cp}] \label{lem203}
Let $m$ be a positive integer. Then there exist uniquely determined polynomials $R_1, \ldots, R_m \in \Z[x]$ with $\emph{deg}(R_k) = k$ and $lc(R_k) = (k-1)!$,
so that
\begin{displaymath}
\log p_n = \log n + \log \log n + \sum_{k=1}^m \frac{(-1)^{k+1}R_k(\log \log n)}{k!\log^kn} + O\left( \frac{(\log \log n)^{m+1}}{\log^{m+1}n} \right).
\end{displaymath}
The polynomials $R_k$ can be computed explicitly. In particular, $R_1(x) = x - 1$, $R_2(x) = x^2 - 4x + 5$ and $R_3(x) = 2x^3 - 15x^2 + 42x - 47$.
\end{lem}

Now, we give another asymptotic formula for the quantity $D(n)$.

\begin{prop} \label{prop204}
Let $r$ be a positive integer and let $T_k(x) = R_k(x) - Q_k(x)$ for $1 \leq k \leq r$. Then, $\emph{deg}(T_k) = k-1$, $lc(T_k) = k!$ and 
\begin{displaymath}
D(n) = 1 + \sum_{k=1}^r \frac{(-1)^{k+1}T_k(\log \log n)}{k!\log^kn} + O\left( \frac{(\log \log n)^r}{\log^{r+1}n} \right).
\end{displaymath}
In particular, $T_1(x) = 1$, $T_2(x) = 2x - 6$ and $T_3(x) = 6x^2 - 42x + 84$.
\end{prop}

\begin{proof}
Let $1 \leq k \leq r$. Since $\text{deg}(Q_k) = \text{deg}(R_k) = k$ and $lc(Q_k) = lc (R_k) = (k-1)!$, we have $\text{deg}(T_k) \leq k-1$. Following Cipolla
\cite[p. 144]{cp}, we write
\begin{displaymath}
Q_k(x) = (k-1)!x^k - a_{k,1}x^{k-1} + \sum_{j=2}^k (-1)^{j}a_{k,j}x^{k-j}
\end{displaymath}
and
\begin{displaymath}
R_k(x) = (k-1)!x^k - b_{k,1}x^{k-1} + \sum_{j=2}^k (-1)^{j}b_{k,j}x^{k-j},
\end{displaymath}
where $a_{i,j}, b_{i,j} \in Z$. 
%Then
%\begin{displaymath}
%T_k(x) = -(b_{k,1}-a_{k.1})x^{k-1} + \sum_{j = 2}^k (-1)^{j}(b_{k,j} - a_{k,j})x^{k-j}.
%\end{displaymath}
By Cipolla \cite[p. 150]{cp}, we have $-(b_{k,1} - a_{k,1}) = k! \neq 0$. Hence, $\text{deg}(T_k) = k-1$ and $lc(T_k) = k!$. Using \eqref{2.4} and the 
definition of $D(n)$, we get
\begin{displaymath}
D(n) = \log p_n - \frac{p_n}{n} + O \left( \frac{1}{\log^{r+1} p_n} \right).
\end{displaymath}
Now we substitute the asymptotic formulae given in Lemma \ref{lem202} and Lemma \ref{lem203} to obtain that
\begin{displaymath}
D(n) = 1 + \sum_{k=1}^{r+1} \frac{(-1)^{k+1}T_k(\log \log n)}{k!\log^kn} + O\left(\frac{1}{\log^{r+1} p_n} \right).
\end{displaymath}
To complete the proof, it suffices to note that $\text{deg}(T_{r+1}) = r$ and $1/\log^{r+1}p_n = O(1/\log^{r+1}n)$.
\end{proof}

\begin{rema}
Proposition \ref{prop204} improves Hassani's \cite{hassani2013} asymptotic formula $D(n) = 1 + O(1/\log n)$.
%Proposition \ref{prop204} implies that 
%\begin{equation}
%D(n) = 1 + \frac{1}{\log n} - \frac{\log \log n - 3}{\log^2 n} + \frac{(\log \log n)^2 - 7 \log \log n + 14}{\log^3 n} + O \left( \frac{(\log \log 
%n)^3}{\log^4n}  \right), \tag{2.9} \label{2.9}
%\end{equation}
%which precises Hassani's \cite{hassani2013} asymptotic formula for $D(n)$. He found that $D(n) = 1 + O(1/\log n)$.
\end{rema}

\subsection{An asymptotic formula for $G_n$}

Next, we derive an asymptotic formula for $G_n$, the geometric mean of the prime numbers. Using the definition of $G_n$ and $D(n)$, we obtain the identity
\begin{equation}
G_n = \frac{p_n}{e^{D(n)}}. \tag{2.5} \label{2.5}
\end{equation}
Proposition \ref{prop201} implies that $\lim_{n \to \infty} D(n) = 1$. Hence,
\begin{equation}
G_n \sim \frac{p_n}{e} \q\q (n \to \infty), \tag{2.6} \label{2.6}
\end{equation}
which was conjectured by Vrba \cite{rivera} in 2010 and proved by S\'{a}ndor and Verroken \cite[Theorem 2.1]{sandor2011} in 2011. Using \eqref{2.5} and 
Proposition \ref{prop201}, we get the following refinement of \eqref{2.6}. Here, the positive integers $k_1, \ldots, k_r$ are defined by the recurrence 
formula \eqref{2.1}.

\begin{prop} \label{prop205}
Let $r$ be a positive integer. Then,
\begin{equation}
G_n = \frac{p_n}{\exp \left( 1 + \frac{k_1}{\log p_n} + \frac{k_2}{\log^2p_n} + \ldots +  \frac{k_r}{\log^r p_n} \right) } + O \left( \frac{p_n}{\log^{r+1}p_n}
\right). \tag{2.7} \label{2.7}
\end{equation}
\end{prop}

\begin{proof}
The claim follows from \eqref{2.5}, Proposition \ref{prop201} and $\exp(c/x) = 1 + O(1/x)$ for every $c \in \R$.
%\begin{displaymath}
%G_n = \frac{p_n}{\exp \left( 1 + \frac{k_1}{\log p_n} + \frac{k_2}{\log^2p_n} + \ldots + \frac{k_r}{\log^r p_n} \right)} \cdot \left( 1 + O \left( 
%\frac{1}{\log^{r+1}p_n}  \right) \right),
%\end{displaymath}
%which completes the proof.
\end{proof}

\begin{rema}
The asymptotic formula \eqref{2.7} was independently found by Kourbatov \cite[Remark (ii)]{kourbatov} in 2016.
\end{rema}

%\subsection{Two asymptotic formulae for $A_n$}

%We start with the following proposition concerning an asymptotic formula for $A_n$, the arithmetic mean of the prime numbers.

%\begin{prop} \label{prop206}
%For each positive integer $m$, we have
%\begin{displaymath}
%A_n = p_n -  \sum_{k=1}^{m-1} (k-1)! \left(1 - \frac{1}{2^k} \right) \frac{p_n^2}{n\log^k p_n} + O \left( \frac{p_n^2}{n\log^m p_n} \right).
%\end{displaymath}
%\end{prop}

%\begin{proof}
%See \cite[Theorem 2]{axler2015}.
%\end{proof}

%Another asymptotic formula for $A_n$ is given as follows.

%\begin{prop} \label{prop207}
%Let $m$ be a positive integer. Then there exist unique monic polynomials $L_s \in \Q[x]$, where $1 \leq s \leq m$ and $\emph{deg}(L_s) = s$, such that
%\begin{displaymath}
%A_n = \frac{n}{2} \left( \log n + \log \log n - \frac{3}{2} + \sum_{s=1}^m \frac{(-1)^{s+1}L_s(\log \log n)}{s\log^s n} \right) + O \left( \frac{n(\log \log 
%n)^{m+1}}{\log^{m+1} n} \right).
%\end{displaymath}
%In particular, $L_1(x) = x - 5/2$ and $L_2(x) = x^2 - 7x + 29/2$.
%\end{prop}

%\begin{proof}
%See \cite[Theorem 1.1]{axler2017}. The polynomials $L_s$ can be computed explicitly by Theorem 2.7 of \cite{axler2017}.
%\end{proof}

\subsection{An asymptotic formula for the ratio of $A_n$ and $G_n$}

Now, we use \eqref{2.5}, Proposition \ref{prop201} and an asymptotic formula for $A_n$ found in \cite[Theorem 2]{axler2015} to prove Theorem \ref{thm101}. Here 
\begin{displaymath}
r_i = (i-1)! \left( 1 - \frac{1}{2^i} \right)
\end{displaymath}
and the positive integers $k_i$ are defined by \eqref{2.1}.

\begin{thm} \label{thm206}
For each positive integer $m$, we have
\begin{displaymath}
\frac{A_n}{G_n} = e\left( \frac{1}{2} + \sum_{i = 1}^m \frac{1}{\log^ip_n} \left( -r_{i+1} + r_i + \sum_{s=1}^{i-1} r_s k_{i-s} \right) \right) \cdot \exp 
\left( \sum_{j = 1}^m\frac{k_j}{\log^jp_n} \right) + O \left( \frac{1}{\log^{m+1}p_n} \right).
\end{displaymath}
\end{thm}

\begin{proof}
By \cite[Theorem 2]{axler2015}, we have
\begin{displaymath}
A_n = p_n -  \sum_{i=1}^{m-1} \frac{r_ip_n^2}{n\log^i p_n} + O \left( \frac{p_n^2}{n\log^m p_n} \right).
\end{displaymath}
Together with \eqref{2.5} and Proposition \ref{prop201}, we obtain
\begin{displaymath}
\frac{A_n}{G_n} = \left( 1 - \sum_{i = 1}^{m+1} \frac{r_ip_n}{n \log^ip_n} + O \left( \frac{p_n}{n \log^{m+2}p_n} \right)\right) \cdot \left( \exp \left( 1 + 
\sum_{j = 1}^m\frac{k_j}{\log^jp_n} \right) + O \left( \frac{1}{\log^{m+1}p_n} \right) \right).
\end{displaymath}
The Prime Number Theorem implies that $p_n \sim n \log p_n$ for $n \to \infty$. This indicates
\begin{equation}
\frac{A_n}{G_n} = e\left( 1 - \sum_{i = 1}^{m+1} \frac{r_ip_n}{n \log^ip_n} + O \left( \frac{1}{ \log^{m+1}p_n} \right)\right) \cdot \exp \left( \sum_{j = 
1}^m\frac{k_j}{\log^jp_n} \right) + O \left( \frac{1}{\log^{m+1}p_n} \right). \tag{2.8} \label{2.8}
\end{equation}
Applying \eqref{2.3} with $x = p_n$ and $r = m-1$ to \eqref{2.8}, we get
%\begin{displaymath}
%\frac{p_n}{n} = \log p_n - 1 - \frac{k_1}{\log p_n} - \ldots - \frac{k_{m-1}}{\log^{m-1} p_n} + O \left( \frac{1}{\log^mp_n} \right).
%\end{displaymath}
%Applying this asymptotic formula to 
\begin{displaymath}
\frac{A_n}{G_n} = e\left( 1 - \sum_{i = 1}^{m+1} \frac{r_i}{\log^ip_n} \left( \log p_n - 1 - \sum_{s=1}^{m-1} \frac{k_s}{\log^s p_n}  \right) \right) \cdot 
\exp \left( \sum_{j = 1}^m\frac{k_j}{\log^jp_n} \right) + O \left( \frac{1}{\log^{m+1}p_n} \right).
\end{displaymath}
Hence,
\begin{align*}
\frac{A_n}{G_n} & = e\left( 1 - \sum_{i = 1}^{m+1} \frac{r_i}{\log^{i-1}p_n} + \sum_{i = 1}^{m+1} \frac{r_i}{\log^ip_n} + \sum_{i = 1}^{m+1} 
\frac{k_1r_i}{\log^{i+1}p_n} + \ldots + \sum_{i = 1}^{m+1} \frac{k_{m-1}r_i}{\log^{m-1+i}p_n}\right) \\
& \p{\q\q} \times \exp \left( \sum_{j = 1}^m\frac{k_j}{\log^jp_n} \right) + O \left( \frac{1}{\log^{m+1}p_n} \right).
\end{align*}
To complete the proof, we separate the terms in the first brace, which are $O(1/\log^{m+1}p_n)$.
%\begin{align*}
%\frac{A_n}{G_n} & = e\left( 1 - \sum_{i = 1}^{m+1} \frac{r_i}{\log^{i-1}p_n} + \sum_{i = 1}^m \frac{r_i}{\log^ip_n} + \sum_{i = 1}^{m-1} 
%\frac{k_1r_i}{\log^{i+1}p_n} + \ldots + \frac{k_{m-1}r_1}{\log^mp_n} + O \left( \frac{1}{ \log^{m+1}p_n} \right)\right) \\
%& \p{\q\q} \times \exp \left( \sum_{j = 1}^m\frac{k_j}{\log^jp_n} \right) + O \left( \frac{1}{\log^{m+1}p_n} \right).
%\end{align*}
%An index offset in the first brace gives
%\begin{align*}
%\frac{A_n}{G_n} = e\left( 1 - r_1 + \sum_{i = 1}^m \frac{1}{\log^ip_n} \left( -r_{i+1} + r_i + \sum_{s=1}^{i-1} r_s k_{i-s} \right) \right) \cdot \exp \left( 
%\sum_{j = 1}^m\frac{k_j}{\log^jp_n} \right) + O \left( \frac{1}{\log^{m+1}p_n} \right),
%\end{align*}
%which completes the proof.
\end{proof}

Now, we use Theorem \ref{thm206} and the asymptotic formula
\begin{displaymath}
\exp \left( \sum_{j = 1}^m\frac{k_j}{\log^jp_n} \right) = \sum_{i=1}^m \frac{1}{i!} \left(\sum_{j = 1}^m\frac{k_j}{\log^jp_n} \right)^i + O \left( 
\frac{1}{\log^{m+1}p_n} \right),
\end{displaymath}
to implement the following \textsc{Maple}-code:
\begin{align*}
& > \texttt{restart:} \vspace{4mm} \\
& \textit{Computation of the values $k_i$:} \\
& > \texttt{for j from 1 to m do} \\
& \p{>>>} \texttt{K[j] := j$\ast$j!-sum(s!$\ast$K[j-s], s=1..j-1):} \\
& \textit{Computation of the values $r_i$:} \\
& > \texttt{for i from 1 to m+1 do} \\
& \p{>>>} \texttt{R[i] := (i-1)!$\ast$(1-1/2$\widehat{\p{a}}\{\texttt{i}\}$):} \\
& \p{>} \texttt{end do:} \\
& > \texttt{AsymptoticExpansion := proc(n) local S1,S2;} \\
% & \p{>>>} \texttt{if n $<= 0$ then ERROR('n_ist_keine_natürliche_Zahl') end if;} \\
&\p{>>>} \texttt{S1 := 1/2} + \texttt{sum}(\texttt{b$\widehat{\p{a}}\{\texttt{w}\}\ast$}(\texttt{-R[w+1]+R[w]+} 
\texttt{sum}(\texttt{R[v]$\ast$K[w-v]}\texttt{,} \texttt{v = 1..(w-1))), w = 1..n);} \\
&\p{>>>} \texttt{S2 := sum}\texttt{(1/t!$\ast$}\texttt{(sum(K[z]$\ast$b$\widehat{\p{a}}\{\texttt{z}\}$, z = 1..n))$\widehat{\p{a}}\{\texttt{t}\}$}\texttt{, 
t = 0..n));} \\
&\p{>>>} \texttt{RETURN(subs(b = 1/log(p$\underline{\p{a}}$n), convert(series(S1$\ast$S2, b,n+1), polynom)));} \\
&\p{>>>} \texttt{end;}
\end{align*}
To give the explicit asymptotic expansion for the ratio of $A_n$ and $G_n$ up to some positive integer $m$, it suffices to write
\begin{align*}
& > \texttt{expand(exp(1)*AsymptoticExpansion(m));} \p{aaaaaaaaaaaaaaaaaaaaaaaaaaaaaaaaaaaaaaaa}
\end{align*}
For instance, we set $m = 5$ to obtain that
\begin{equation}
\frac{A_n}{G_n} = \frac{e}{2} + \frac{e}{4\log p_n} + \frac{e}{\log^2 p_n} + \frac{61e}{12\log^3p_n} + \frac{1463e}{48\log^4p_n} + \frac{100367e}{480\log^5p_n} 
+ O \left( \frac{1}{\log^6p_n} \right). \tag{2.9} \label{2.9}
\end{equation}
One of Hassani's results \cite[p. 1602]{hassani2013} is that $A_n/G_n = e/2 + O(1/\log n)$. The asymptotic expansion given in \eqref{2.9} precises this result.

%$m = 2$ in Theorem \ref{thm206}, we get the following more accurate asymptotic formula.

% \begin{kor} \label{kor207}
% We have
% \begin{displaymath}
% \frac{A_n}{G_n} = \frac{e}{2} + \frac{e}{4\log p_n} + \frac{e}{\log^2 p_n} + O \left( \frac{1}{\log^3p_n} \right).
% \end{displaymath}
% \end{kor}
% 
% \begin{proof}
% We set $m = 2$ in Theorem \ref{thm206} to get
% \begin{displaymath}
% \frac{A_n}{G_n} = e\left( \frac{1}{2} - \frac{1}{4\log p_n} - \frac{1}{2\log^2 p_n} \right) \cdot \exp \left( \sum_{j = 1}^2\frac{k_j}{\log^jp_n} \right) + O 
% \left( \frac{1}{\log^3p_n} \right).
% \end{displaymath}
% Together with $\exp(1/x) = 1 + 1/x + 1/(2!x^2) + O(1/x^3)$, we obtain
% %\begin{displaymath}
% %\exp \left( \sum_{j = 1}^2\frac{k_j}{\log^jp_n} \right) = 1 + \frac{1}{\log p_n} + \frac{7}{2 \log^2 p_n} + O \left( \frac{1}{\log^3p_n} \right),
% %& = 1 + \sum_{s=1}^2 \frac{1}{s!} \left( \sum_{j = 1}^2\frac{k_j}{\log^jp_n} \right)^s + O \left( \frac{1}{\log^3p_n} \right) \\ 
% %\end{displaymath}
% %since $k_1 = 1$ and $k_2 = 3$. Applying this to \eqref{2.16}, we obtain that
% \begin{displaymath}
% \frac{A_n}{G_n} = e\left( \frac{1}{2} - \frac{1}{4\log p_n} - \frac{1}{2\log^2 p_n} \right) \cdot \left( 1 + \frac{1}{\log p_n} + \frac{7}{2 \log^2 p_n} + O 
% \left( \frac{1}{\log^3p_n} \right) \right) + O \left( \frac{1}{\log^3p_n} \right),
% \end{displaymath}
% which completes the proof.
% \end{proof}

\subsection{An asymptotic formula for the quantity $\log(1+2R(n)/p_n)$}

Finally, we derive an asymptotic formula for $\log(1+2R(n)/p_n)$ for $n \to \infty$, where
\begin{displaymath}
R(n) = \frac{1}{n}\sum_{k \leq n} p_k - \frac{p_n}{2}.
\end{displaymath}

\begin{prop} \label{prop207}
We have
\begin{displaymath}
\log \left( 1 + \frac{2R(n)}{p_n} \right) \sim -\frac{1}{2 \log n} \q\q (n \to \infty).
\end{displaymath}
\end{prop}

\begin{proof}
We have $p_n \sim n \log n$ and, by \cite{axler2017}, $R(n) \sim - n/4$ for $n \to \infty$. Hence, $\log(1+2R(n)/p_n) \sim \log( 1 - 1/(2 \log n))$ for $n \to 
\infty$. Since $\log (1-1/(2x)) \sim -1/(2x)$ for $x \to \infty$, the proposition is proved.
\end{proof}

\begin{rema}
At the end of Section 5, we give a more accurate asymptotic formula for $\log(1+2R(n)/p_n)$.
\end{rema}

\section{New estimates for the quantity $D(n)$}

After giving two asymptotic formulae for the quantity $D(n)$ in Subsection 2.1, we are interested in finding 
some explicit estimates for $D(n)$. 

\subsection{Explicit estimates for $D(n)$ in terms of $p_n$}

In this subsection, we give some explicit estimates for $D(n)$ in terms of $p_n$, which corresponds to the first three terms of the asymptotic expansion given 
in Proposition \ref{prop201}. We start with the following lower bound.

\begin{prop} \label{prop301}
For every positive integer $n \geq 218$, we have
\begin{equation}
D(n) > 1 + \frac{1}{\log p_n} + \frac{2.7}{\log^2p_n}. \tag{3.1} \label{3.1}
\end{equation}
\end{prop}

\begin{proof}
Substituting $x = p_n$ in \cite[Corollary 3.9]{axler20172}
% it is shown that the inequality
%\begin{displaymath}
%\pi(x) > \frac{x}{\log x - 1 - \frac{1}{\log x} - \frac{2.85}{\log^2 x}}
%\end{displaymath}
%holds for every $x \geq 38099531 = p_{2324692}$. Setting $x = p_n$,
we get that the inequality
\begin{equation}
\log p_n  > \frac{p_n}{n} + 1 + \frac{1}{\log p_n} + \frac{2.85}{\log^2p_n} \tag{3.2} \label{3.2}
\end{equation}
is fulfilled for every positive integer $n \geq 2\,324\,692$. Next, in \cite[Theorem 1.1]{axler20172} it is shown that $\vartheta(x) < x + 0.15x/\log^3x$ for 
every $x > 1$. Together with \eqref{3.2} and the definition of $D(n)$, we get
\begin{equation}
D(n) > 1 + \frac{1}{\log p_n} + \frac{2.85}{\log^2p_n} - \frac{0.15p_n}{n \log^3p_n} \tag{3.3} \label{3.3}
\end{equation}
for every positive integer $n \geq 2\,324\,692$. Setting $x = p_n$ in \cite[Corollary 1]{rosser1962}, we obtain that
\begin{equation}
p_n \leq n \log p_n \tag{3.4} \label{3.4}
\end{equation}
for every positive integer $n \geq 7$. Applying \eqref{3.4} to \eqref{3.3}, we get that the inequality \eqref{3.1} holds for every positive integer $n \geq 
2\,324\,692$. We conclude by direct computation.
\end{proof}

In the following proposition, we give two lower bounds for $D(n)$, which improve the inequality \eqref{3.1} for all sufficiently large values of $n$.

\begin{prop} \label{prop302}
For every positive integer $n$, we have
\begin{equation}
D(n) > 1 + \frac{1}{\log p_n} + \frac{3}{\log^2p_n} - \frac{187}{\log^3 p_n} \tag{3.5} \label{3.5}
\end{equation}
and
\begin{equation}
D(n) > 1 + \frac{1}{\log p_n} + \frac{3}{\log^2p_n} + \frac{13}{\log^3p_n} - \frac{1160159}{\log^4 p_n}. \tag{3.6} \label{3.6}
\end{equation}
\end{prop}

\begin{proof}
We start with the proof of \eqref{3.5}. By \cite[Proposition 2.5]{axler20172}, we have $|\vartheta(x) - x| < 100x/\log^4x$ for every $x \geq 70\,111 = 
p_{6\,946}$. Furthermore, in \cite[Proposition 3.10]{axler20172} it is found that the inequality
\begin{displaymath}
\pi(x) > \frac{x}{\log x - 1 - \frac{1}{\log x} - \frac{3}{\log^2 x} + \frac{87}{\log^3 x}}
\end{displaymath}
holds for every $x \geq 19\,423$. Similar to the proof of Proposition \ref{prop301}, we get that the inequality \eqref{3.5} is fulfilled for every positive 
integer $n \geq 6\,946$. For smaller values of $n$, we use a computer.

Next we give the proof of \eqref{3.6}. In \cite[Corollary 2.2]{axler20173}, it is shown that the inequality $|\vartheta(x) - x| < 580115x/\log^5x$ holds for 
every $x \geq 2$. Further, we found in the proof of \cite[Theorem 1.1]{axler20173} that
\begin{displaymath}
\pi(x) > \frac{x}{\log x - 1 - \frac{1}{\log x} - \frac{3}{\log^2 x} - \frac{13}{\log^3 x} + \frac{580044}{\log^4x}}
\end{displaymath}
for every $x \geq 10^{13}$. Similar to the proof of Proposition \ref{prop301}, we get that the inequality \eqref{3.6} holds for every positive integer $n \geq 
\pi(10^{13}) + 1$. For smaller values of $n$, we use \eqref{3.5}.
\end{proof}

Since $k_1 = 1$ and $k_2 = 3$, Proposition \ref{prop201} implies that there is a smallest positive integer $N_0$ so that
\begin{equation}
D(n) > 1 + \frac{1}{\log p_n} + \frac{3}{\log^2p_n} \tag{3.7} \label{3.7}
\end{equation}
for every positive integer $n \geq N_0$. In the following corollary, we make a first progress in finding this $N_0$.

\begin{kor} \label{kor303}
The inequality \eqref{3.7} holds for every positive integer $n$ satisfying $264 \leq n \leq \pi(10^{19}) = 234\,057\,667\,276\,344\,607$ and $n \geq 
\pi(e^{1160159/13}) + 1$.
\end{kor}

\begin{proof}
The inequality \eqref{3.6} implies the validity of \eqref{3.7} for every positive integer $n \geq \pi(e^{1160159/13}) + 1$. So, it suffices to prove that the 
inequality \eqref{3.7} holds for every positive integer $n$ such that $264 \leq n \leq \pi(10^{19})$. By \cite[Theorem 1.1]{axler20173}, we have
\begin{equation}
\pi(x) > \frac{x}{\log x - 1 - \frac{1}{\log x} - \frac{3}{\log^2 x}} \tag{3.8} \label{3.8}
\end{equation}
for every $x$ satisfying $65\,405\,887 \leq x \leq 5.5 \cdot 10^{25}$ and $x \geq e^{580044/13}$. Büthe \cite[Theorem 2]{buethe} found that $\vartheta(x) < x$ 
for every $x$ such that $1 \leq x \leq 10^{19}$ and together with \eqref{3.8}, we get, similar to the proof of Proposition \ref{prop301}, that the inequality 
\eqref{3.7} holds for every positive integer $n$ with $\pi(65\,405\,887) \leq n \leq \pi(10^{19})$. Finally, we check the remaining cases with a computer.
\end{proof}

Based on Corollary \ref{kor303} we state the following conjecture.

\begin{conj} \label{conj304}
The inequality \eqref{3.7} holds for every positive integer $n \geq 264$.
\end{conj}

Next, we establish some explicit upper bounds for $D(n)$ in terms of $p_n$. From Proposition \ref{prop201} follows that for each $\varepsilon > 0$ there is a 
positive integer $N_1 = N_1(\varepsilon)$, such that
\begin{displaymath}
D(n) < 1 + \frac{1}{\log p_n} + \frac{3+ \varepsilon}{\log^2p_n}
\end{displaymath}
for every positive integer $n \geq N_1$. We find the following.

\begin{prop} \label{prop305}
For every positive integer $n \geq 74\,004\,585$, we have
\begin{equation}
D(n) < 1 + \frac{1}{\log p_n} + \frac{3.84}{\log^2p_n}, \tag{3.9} \label{3.9}
\end{equation}
and for every positive integer $n$, we have
\begin{equation}
D(n) < 1 + \frac{1}{\log p_n} + \frac{3}{\log^2p_n} + \frac{213}{\log^3p_n}. \tag{3.10} \label{3.10}
\end{equation}
\end{prop}

\begin{proof}
We start with the proof of \eqref{3.9}. First, we consider the case where $n \geq 841\,508\,302$. From \cite[Corollary 3.3]{axler20172} follows that
\begin{equation}
\log p_n  < \frac{p_n}{n} + 1 + \frac{1}{\log p_n} + \frac{3.69}{\log^2p_n}. \tag{3.11} \label{3.11}
\end{equation}
Furthermore, by \cite[Theorem 1.1]{axler20172} we have $\vartheta(x) > x - 0.15x/\log^3x$ for every $x \geq 19\,035\,709\,163 = p_{841\,508\,302}$. 
Together with the definition of $D_n$ and the inequality \eqref{3.11} we obtain that
\begin{displaymath}
D(n) < 1 + \frac{1}{\log p_n} + \frac{3.69}{\log^2p_n} + \frac{0.15p_n}{n\log^3p_n}.
\end{displaymath}
Now we use \eqref{3.4} to get that the inequality \eqref{3.9} holds for every positive integer $n \geq 841\,508\,302$. For smaller values of $n$, we check the 
required inequality with a computer.

Next, we establish the inequality \eqref{3.10}. In \cite[Proposition 3.5]{axler20172}, it is shown that 
\begin{displaymath}
\pi(x) < \frac{x}{\log x - 1 - \frac{1}{\log x} - \frac{3}{\log^2x} - \frac{113}{\log^3 x}}
\end{displaymath}
for every $x \geq 41$. By \cite[Proposition 2.5]{axler20172}, we have $\vert \vartheta(x) - x \vert < 100 x/\log^4 x$ for every $x \geq 70\,111$. Now we argue 
as in the proof of Proposition \ref{prop302}. For the remaining cases, we use a computer.
\end{proof}

\subsection{Explicit estimates for $D(n)$ in terms of $n$}

Since the computation of $p_n$ is difficult for large $n$, the estimates for $D(n)$ obtained in Subsection 3.1 are ineffective for large $n$. Hence, we are
interested in estimates for $D(n)$ in terms of $n$. First, we note that Proposition \ref{prop204} implies that
\begin{equation}
D(n) = 1 + \frac{1}{\log n} - \frac{\log \log n - 3}{\log^2 n} + O \left( \frac{(\log \log n)^2}{\log^3n}  \right). \tag{3.12} \label{3.12}
\end{equation}
The goal of this subsection is to find explicit estimates for $D(n)$ in the direction of \eqref{3.12}. We start with lower bounds. Hassani \cite[Proposition 
1.6]{hassani2013} showed that the inequality $D(n) > 1 - 17/(5\log n)$ is valid for every positive integer $n \geq 2$. We give the following refinement.

\begin{prop} \label{prop306}
For every positive integer $n \geq 591$, we have
\begin{equation}
D(n) > 1 + \frac{1}{\log n} - \frac{\log \log n - 2.5}{\log^2 n}. \tag{3.13} \label{3.13}
\end{equation}
\end{prop}

\begin{proof}
We denote the right-hand side of \eqref{3.13} by $f(n)$. First, let $n$ be a positive integer with $n \geq \pi(10^{19}) = 234\,057\,667\,276\,344\,607$. By 
\cite[Corollary 3.3]{axler20174}, we have
\begin{equation}
\frac{1}{\log p_n} \geq \frac{1}{\log n} - \frac{\log \log n}{\log^2 n} + \frac{(\log \log n)^2 - \log \log n + 1}{\log^2n \log p_n}, \tag{3.14} \label{3.14}
\end{equation}
which implies that the weaker inequality
\begin{equation}
\frac{1}{\log p_n} \geq \frac{1}{\log n} - \frac{\log \log n}{\log^2n}  \tag{3.15} \label{3.15}
\end{equation}
also holds. After combining \eqref{3.14} and \eqref{3.15}, we get
\begin{equation}
\frac{1}{\log p_n} \geq \frac{1}{\log n} - \frac{\log \log n}{\log^2 n} + \frac{(\log \log n)^2 - \log \log n + 1}{\log^2n} \left( \frac{1}{\log n} - 
\frac{\log \log n}{\log^2 n} \right). \tag{3.16} \label{3.16}
\end{equation}
Together with \eqref{3.1} and \eqref{3.15}, we get
\begin{displaymath}
D(n) > f(n) + \frac{0.2}{\log^2 n} + \frac{(\log \log n)^2 - 5.6 \log \log n + 1}{\log^3n} - \frac{(\log \log n)^3 - 3.3 (\log \log n)^2 + \log \log 
n}{\log^4n},
\end{displaymath}
which completes the proof for every positive integer $n \geq \pi(10^{19})$.
%where $f(n)$ denotes the right-hand site of \eqref{3.9}. Since
%\begin{displaymath}
% \frac{0.16}{\log^2 x} + \frac{(\log \log x)^2 - 5.6 \log \log x + 1}{\log^3x} - \frac{(\log \log x)^3 - 3.3 (\log \log x)^2 + \log \log x}{\log^4x} > 0
%\end{displaymath}
%for every $x \geq 1.3 \cdot 10^{17}$, the required inequality holds for every $n > \pi(10^{19}) \geq 1.3 \cdot 10^{17}$.

Similarly to the case $n \geq \pi(10^{19})$, we combine \eqref{3.15}, \eqref{3.16} and Corollary \ref{kor303} to get that
\begin{displaymath}
D(n) > f(n) + \frac{0.5}{\log^2 n} + \frac{(\log \log n)^2 - 7 \log \log n + 1}{\log^3n} - \frac{(\log \log n)^3 - 4 (\log \log n)^2 + \log \log n}{\log^4n}
\end{displaymath}
for every positive integer $n$ such that $264 \leq n \leq \pi(10^{19})-1$, which implies that the required inequality holds for every positive integer $n$ 
such that $2\,426\,927\,728 \leq n \leq \pi(10^{19})-1$. We verify the remaining cases with a computer.
\end{proof}

We get the following corollary.

\begin{kor} \label{kor307}
For every $\alpha < 1$ there exists a positive integer $n_0 = n_0(\alpha)$ so that $D(n) > 1 + \alpha/\log n$ for every positive integer $n \geq n_0$.
\end{kor}

\begin{rema}
Hassani \cite[Conjecture 1.7]{hassani2013} conjectured that there exist a real number $\beta$ with $0 < \beta < 5.25$ and a positive integer $n_0$, such that 
the inequality $D(n) > 1 + \beta/\log n$ is valid for every positive integer $n \geq n_0$. Corollary \ref{kor307} proves this conjecture. The inequality 
\eqref{3.13} implies that $D(n) > 1$ for every positive integer $n \geq 591$. A computer check shows that the last inequality also holds for every positive 
integer $n$ with $10 \leq n \leq 591$. So, the inequality $D(n) > 1$ holds for every positive integer $n \geq 10$, which was conjectured by Hassani 
\cite[Conjecture 1.7]{hassani2013}.
\end{rema}

Next, we establish some new upper bounds for $D(n)$ in terms of $n$. Using estimates for the $n$-th prime number and Chebyshev's $\vartheta$-function, Hassani 
\cite[Proposition 1.6]{hassani2013} found that $D(n) < 1 + 21/(4\log n)$ for every positive integer $n \geq 2$. We give the following improvement of Hassani's 
upper bound. 

\begin{prop} \label{prop308}
For every positive integer $n \geq 2$, we have
\begin{displaymath}
D(n) < 1 + \frac{1}{\log n} - \frac{\log \log n - 4.2}{\log^2 n}.
\end{displaymath}
In particular, for every $\beta \geq 1$ there exists a positive integer $n_1 = n_1(\beta)$ so that $D(n) < 1 + \beta/\log n$ for every positive integer $n \geq 
n_1$.
\end{prop}

\begin{proof}
By \cite[Corollary 3.6]{axler20174}, we have
\begin{displaymath}
\frac{1}{\log p_n} \leq \frac{1}{\log n} - \frac{\log \log n}{\log^2 n} + \frac{(\log \log n)^2 - \log \log n + 1}{\log^2 n \log p_n} + \frac{P_8(\log \log 
n)}{2 \log^3n\log p_n} - \frac{P_9(\log \log n)}{2 \log^4n\log p_n}
\end{displaymath}
for every positive integer $n \geq 2$, where $P_8(x) = 3x^2 - 6x + 5.2$ and $P_9(x) =  x^3 - 6x^2 + 11.4x - 4.2$. Since $P_9(x) > 0$ for every $x \geq 0.5$, we 
get
\begin{equation}
\frac{1}{\log p_n} \leq \frac{1}{\log n} - \frac{\log \log n}{\log^2 n} + \frac{(\log \log n)^2 - \log \log n + 1}{\log^3 n} + \frac{P_8(\log \log n)}{2 \log^4
n} \tag{3.17} \label{3.17}
\end{equation}
for every positive integer $n \geq 6$. Together with Proposition \ref{prop305} and the inequality $3.84/\log^2 p_n \leq 3.84/\log^2 n$, we obtain that the 
inequality
\begin{equation}
D(n) < 1 + \frac{1}{\log n} - \frac{\log \log n - 3.84}{\log^2 n} + \frac{(\log \log n)^2 - \log \log n + 1}{\log^3 n} + \frac{P_8(\log \log n)}{2 \log^4 n} 
\tag{3.18} \label{3.18}
\end{equation}
holds for every positive integer $n \geq $. Notice that the inequality
\begin{equation}
\frac{(\log \log x)^2 - \log \log x + 1}{\log^3 x} + \frac{P_8(\log \log x)}{2 \log^4 x} < \frac{0.36}{\log^2 x} \tag{3.19} \label{3.19}
\end{equation}
holds for every $x \geq 1\,499\,820\,545$. Applying this to \eqref{3.18}, we get that the required inequality for every positive integer $n \geq 
1\,499\,820\,545$.
%. Notice that
%\begin{displaymath}
%\frac{(\log \log x)^2 - \log \log x + 1}{\log^3 x} + \frac{P_8(\log \log x)}{2 \log^4 x} < \frac{0.38}{\log^2 x}
%\end{displaymath}
%for every $x \geq 56615486$. Applying this inequality to \eqref{3.20}, the claim follows for every $n \geq 56615486$. 
Finally, we use a computer to check that the required inequality also holds for smaller values of $n$.
\end{proof}

%%%%%%%%%%%%%%%%%%%%%%%%%%%%%%%%%%%%%%%%%%%%%%%%%

\section{New estimates for the geometric mean of the prime numbers}

In the following, we use the identity \eqref{2.5}; i.e. $G_n = p_n/e^{D(n)}$, and the explicit estimates for $D(n)$ obtained in the previous section to find
new bounds for $G_n$, the geometric mean of the prime numbers. First, we notice that \eqref{2.5} and the inequality $D(n) > 1$, which holds for every positive 
integer $n \geq 10$, imply $G_n < p_n/e$ for every positive integer $n \geq 10$, which was already proved by Panaitopol \cite{panaitopol} in 1999. In the 
direction of Proposition \ref{prop205}, Kourbatov \cite[Theorem 2]{kourbatov} used explicit estimates for the prime counting function $\pi(x)$ and Chebyshev's 
$\vartheta$-function to show that the inequality
\begin{displaymath}
G_n < \frac{p_n}{\exp(1 + \frac{1}{\log p_n} + \frac{1.62}{\log^2 p_n})}
%G_n < \frac{p_n}{\exp(1 + 1/\log p_n + 1.62/\log^2 p_n)}
\end{displaymath}
is fulfilled for every prime number $p_n \geq 32\,059$; i.e. for every positive integer $n \geq 3\,439$. Actually, this inequality also holds for every $92 
\leq n \leq 3\,438$ as well. In the next proposition, we give some sharper estimates for $G_n$.

\begin{prop} \label{prop401}
For every positive integer $n \geq 218$, we have
\begin{displaymath}
G_n < \frac{p_n}{\exp(1 + \frac{1}{\log p_n} + \frac{2.7}{\log^2 p_n})}.
\end{displaymath}
For every positive integer $n$, we have
\begin{displaymath}
G_n < \frac{p_n}{\exp(1 + \frac{1}{\log p_n} + \frac{3}{\log^2 p_n} - \frac{187}{\log^3 p_n})}
\end{displaymath}
and
\begin{displaymath}
G_n < \frac{p_n}{\exp(1 + \frac{1}{\log p_n} + \frac{3}{\log^2 p_n} + \frac{13}{\log^3 p_n} - \frac{1160159}{\log^4 p_n})}.
\end{displaymath}
\end{prop}

\begin{proof}
The first inequality is a direct consequence of \eqref{2.5} and Proposition \ref{prop301}. Further, we apply the inequalities obtained in Proposition 
\ref{prop302} to the identity \eqref{2.5} and get the remaining inequalities.
\end{proof}

Next, we use Corollary \ref{kor303} to get the following upper bound.

\begin{prop} \label{prop402}
For every positive integer $n$ satisfying $264 \leq n \leq \pi(10^{19}) = 234\,057\,667\,276\,344\,607$ and $n \geq \pi(e^{1160159/13}) + 1$, we have
\begin{displaymath}
G_n < \frac{p_n}{\exp(1 + \frac{1}{\log p_n} + \frac{3}{\log^2 p_n})}.
\end{displaymath}
\end{prop}

\begin{proof}
We combine \eqref{2.5} with Corollary \ref{kor303}.
\end{proof}

Proposition \ref{prop205} and the Prime Number Theorem imply that
\begin{equation}
G_n = \frac{p_n}{e} + O(n), \tag{4.1} \label{4.1}
\end{equation}
which was already obtained by Hassani \cite[p. 1602]{hassani2013} in 2013. In order to find new upper bounds for $G_n$ in the direction of the asymptotic 
formula \eqref{4.1}, we first notice the following result.

\begin{prop} \label{prop403}
For every positive integer $n \geq 47$, then
\begin{displaymath}
G_n < \frac{p_n}{e}\left( 1 - \frac{1}{\log p_n}\right).
\end{displaymath}
\end{prop}

\begin{proof}
Using Proposition \ref{prop401} and the inequality $e^x \geq 1+x$, which holds for every real $x$, we get
\begin{displaymath}
G_n < \frac{p_n}{e} \left( 1 -  \frac{\log p_n + 2.7}{\log^2p_n + \log p_n + 2.7} \right)
\end{displaymath}
for every positive integer $n \geq 218$. Since $\log p_m > 2.7/1.7$ for every positive integer $m \geq 3$, we obtain the required inequality for every positive 
integer $n \geq 218$. For the remaining cases of $n$ we use a computer.
\end{proof}

In the direction of \eqref{4.1}, we find the following upper bound for $G_n$ .

\begin{kor} \label{kor404}
For every positive integer $n \geq 31$, we have
\begin{displaymath}
G_n < \frac{p_n}{e} - \frac{n}{e}\left( 1 - \frac{1}{\log p_n} - \frac{1}{\log^2p_n} - \frac{3.69}{\log^3p_n}\right).
\end{displaymath}
In particular, for every real $\gamma$ with $0 < \gamma < 1/e$ there is a positive integer $n_2 = n_2(\gamma)$ so that $G_n < p_n/e - \gamma n$ for every 
positive integer $n \geq 
n_2$.
\end{kor}

\begin{proof}
We use \eqref{3.11} and Proposition \ref{prop403} to get that the required inequality holds for every positive integer $n \geq 456\,441\,574$. We conclude by 
direct computation.
\end{proof}

\begin{rema}
The second part of Corollary \ref{kor404} proves a conjecture stated by Hassani \cite[Conjecture 4.3]{hassani2013}.
\end{rema}

Next, we find new lower bounds for $G_n$. In view of Proposition \ref{prop205}, Kourbatov \cite[Theorem 2]{kourbatov} used explicit estimates for the prime 
counting function $\pi(x)$ and Chebyshev's $\vartheta$-function to find that the inequality
\begin{displaymath}
G_n > \frac{p_n}{\exp(1 + \frac{1}{\log p_n} + \frac{4.83}{\log^2 p_n})}
\end{displaymath}
holds for every positive integer $n \geq 3\,439$. In the next proposition, we give two sharper lower bounds.

\begin{prop} \label{prop405}
For every positive integer $n \geq 74\,004\,585$, we have
\begin{equation}
G_n > \frac{p_n}{\exp(1 + \frac{1}{\log p_n} + \frac{3.84}{\log^2 p_n})}, \tag{4.2} \label{4.2}
\end{equation}
and for every positive integer $n$, we have
\begin{displaymath}
G_n > \frac{p_n}{\exp(1 + \frac{1}{\log p_n} + \frac{3}{\log^2p_n} + \frac{213}{\log^3p_n})}.
\end{displaymath}
\end{prop}

\begin{proof}
We use \eqref{2.5} and Proposition \ref{prop305} to obtain the required inequalities.
\end{proof}

In order to derive a lower bound for $G_n$ in the direction of \eqref{4.1}, we first establish the following result.

\begin{prop} \label{prop406}
For every positive integer $n$, we have
\begin{displaymath}
G_n > \frac{p_n}{e}\left( 1 - \frac{1}{\log p_n} - \frac{4.74}{\log^2p_n}\right).
\end{displaymath}
\end{prop}

\begin{proof}
First, we consider the case where $n \geq 883\,051\,281 = \pi(e^{23.72}) + 1$. It is easy to see that
\begin{equation}
e^t < 1 + t + \frac{2t^2}{3} \tag{4.3} \label{4.3}
\end{equation}
for every $t < \log(4/3)$. Hence, we obtain
\begin{displaymath}
\exp \left( \frac{1}{x} + \frac{3.84}{x^2} \right) < 1 + \frac{1}{x} + \frac{13.52}{3x^2} + \frac{5.12}{x^3} + \frac{9.8304}{x^4}
\end{displaymath}
for every $x \geq 6$. Now, if $x \geq 23.72$, then $5.12/x + 9.8304/x^2 < 0.23333$ and we get
\begin{equation}
\exp \left( \frac{1}{x} + \frac{3.84}{x^2} \right) < 1 + \frac{1}{x} + \frac{4.74}{x^2} \tag{4.4} \label{4.4}
\end{equation}
for every $x \geq 23.72$. Since $\log p_n \geq 23.72$, it follows from \eqref{4.2} and the inequality \eqref{4.4} that
\begin{displaymath}
G_n > \frac{p_n}{e}\left( 1 - \frac{\log p_n + 4.74}{\log^2 p_n + \log p_n + 4.74} \right).
\end{displaymath}
Since the right-hand side of the last inequality is greater then the right-hand side of the required inequality, the corollary is proved for every positive 
integer $n \geq 883\,051\,281$. A computer check shows that the asserted inequality holds for every positive integer $n$ with $1 \leq n \leq 64\,881\,103$ as 
well.
\end{proof}

In view of \eqref{4.1}, Hassani \cite[Corollary 4.2]{hassani2013} found that
\begin{displaymath}
G_n > \frac{p_n}{e} - 2.37n
\end{displaymath}
for every positive integer $n$. The following corollary improves this inequality.

\begin{kor} \label{kor407}
For every positive integer $n \geq 3$, we have
\begin{displaymath}
G_n > \frac{p_n}{e} - \frac{n}{e}\left( 1 + \frac{3.74}{\log p_n} - \frac{5.74}{\log^2p_n} - \frac{7.59}{\log^3p_n}\right).
\end{displaymath}
In particular, for every $\delta > 1/e$ there is a positive integer $n_3 = n_3(\delta)$ so that $G_n > p_n/e - \delta n$ for every positive integer $n \geq 
n_3$.
\end{kor}

\begin{proof}
First, we consider the case where $n \geq 2\,324\,692$. We use \eqref{3.2} and the inequality obtained in Proposition \ref{prop406} to get that
\begin{equation}
G_n > \frac{p_n}{e} - \frac{n}{e}\left( 1 - \frac{1}{\log p_n} - \frac{1}{\log^2p_n} -  \frac{2.85}{\log^3p_n}\right) - \frac{4.74p_n}{e\log^2 p_n}. \tag{4.5}
\label{4.5}
\end{equation}
The inequality \eqref{3.2} implies that
\begin{displaymath}
-\frac{4.74p_n}{e\log^2 p_n} > - \frac{4.74n}{e \log p_n} \left( 1 - \frac{1}{\log p_n} - \frac{1}{\log^2 p_n} \right).
\end{displaymath}
We apply this inequality to \eqref{4.5} and obtain the required inequality. For every positive integer $n$ satisfying $3 \leq n \leq 2\,324\,692$ we check the 
required inequality with a 
computer.
\end{proof}

\begin{rema}
Corollary \ref{kor404} and Corollary \ref{kor407} yield a more accurate asymptotic formula for $G_n$ than in \eqref{4.1}, namely that
\begin{displaymath}
G_n = \frac{p_n}{e} - \frac{n}{e} + O \left( \frac{n}{\log p_n} \right).
\end{displaymath}
\end{rema}

\section{New estimates for the quantity $\log(1 + 2R(n)/p_n)$}

First, we recall the definition of $R(n)$, namely $R(n) = \sum_{k \leq n}p_k/n - p_n/2$. Hassani \cite[Corollary 1.5]{hassani2013} proved that
\begin{equation}
- \frac{15}{2 \log n} < \log \left( 1 + \frac{2R(n)}{p_n} \right)  < - \frac{5}{36 \log n}, \tag{5.1} \label{5.1}
\end{equation}
where the left-hand side inequality holds for every positive integer $n \geq 2$, and the right-hand side inequality holds for every positive integer $n \geq 
10$. In Proposition \ref{prop207}, we gave a more suitable approximation for the quantity $\log(1 + 2R(n)/p_n)$ for $n \to \infty$. In the direction of  this 
approximation, we improve the inequalities found in \eqref{5.1}. The first proposition is about a lower bound for $\log(1 + 2R(n)/p_n)$.

\begin{prop} \label{prop501}
For every positive integer $n \geq 26\,220$, we have
\begin{equation}
\log \left( 1 + \frac{2R(n)}{p_n} \right) > - \frac{1}{2 \log n} + \frac{\log \log n - 2.25}{2 \log^2 n} - \frac{(\log \log n)^2 - 4.5 \log \log n + 
22.51/3}{2 \log^3 n}. \tag{5.2} \label{5.2}
\end{equation}
\end{prop}

\begin{proof}
First, we note a result proved by Dusart \cite{dusart99} concerning a lower bound for $p_n$, namely that
\begin{equation}
p_n \geq r(n) \tag{5.3} \label{5.3}
\end{equation}
for every positive integer $n \geq 2$, where $r(x) = x(\log x + \log \log x - 1)$. We set
\begin{displaymath}
s_1(x) = - \frac{x}{4} - \frac{x}{4\log x} + \frac{x(\log \log x - 4.42)}{4 \log^2 x}
\end{displaymath}
and by \cite[Theorem 1.8]{axler2017} and Hassani \cite[Corollary 1.5]{hassani2013}, we obtain that
\begin{equation}
s_1(n) < R(n) < 0 \tag{5.4} \label{5.4}
\end{equation}
for every positive integer $n \geq 256\,376$. Now, we define $h(x) = \log(1 + 2s_1(x)/r(x))$ and show that $h(x)$ is greater than the right-hand side of 
\eqref{5.2}. For this, we set
\begin{align*}
f(y) & = (2\log^3y - 13\log^2y + 33.09 \log y - 29.1575)y^3  \\
& \p{\q\q} + (1.5 \log^4y - 11.5\log^3y + 34.63 \log^2y - 41.1575\log y + 9.9075)y^2 \\ 
& \p{\q\q} + (-0.5 \log^3 y + 1.52 \log^2y + 3.59 \log y - 17.01605)y \\
& \p{\q\q} + 0.75 \log^4 y - 7.94 \log^3 y + 31.07 \log^2y - 53.72605 \log y + 29.84605
\end{align*}
and
\begin{displaymath}
g(y) = y^3 + y^2\log y - 1.5y^2 - 0.5y + 0.5\log y - 2.21.
\end{displaymath}
%Since $2\log^3y - 13\log^2y + 33.09 \log y - 29.1575 \geq 3$ for every $y \geq e^{2.4}$ and $1.5 \log^4y - 11.5\log^3y + 34.63 \log^2y - 41.1575\log y + 
%9.9075 \geq 0$ for every $y \geq 10$, we get
%\begin{align*}
%f(y) & \geq 3y^3 + (-0.5 \log^3 y + 1.52 \log^2y + 3.59 \log y - 17.01605)y \\
%& \p{\q\q} + 0.75 \log^4 y - 7.94 \log^3 y + 31.07 \log^2y - 53.72605 \log y + 29.84605
%\end{align*}
%for every $y \geq e^{2.4}$. Notice that $y^2 \geq 2\log^3 y$ for every $y \geq 1$. Hence,
%\begin{align*}
%f(y) & \geq (11.5 \log^3 y + 1.52 \log^2y + 3.59 \log y - 17.01605)y \\
%& \p{\q\q} + 0.75 \log^4 y - 7.94 \log^3 y + 31.07 \log^2y - 53.72605 \log y + 29.84605
%\end{align*}
%for every $y \geq e^{2.4}$. Now we apply the inequality $y \geq \log y$, which holds for every $y > 0$, to get
%\begin{displaymath}
%f(y) \geq 12.25\log^4 y - 6.42 \log^3y + 34.66 \log^2 y - 70.7421\log y + 29.84605
%\end{displaymath}
%for every $y \geq e^{2.4}$. Now, it is clear that the right-hand side of the last inequality is positive for every $y \geq e^2$. Hence $f(y) > 0$ for every $y
%\geq e^{2.4}$. Similarly, we get that $g(y) > 0$ for every $y \geq e^{2.4}$. Therefore,
It is easy to see that $f(y)$ and $g(y)$ are positive for every $y \geq e^{2.4}$. Hence
\begin{displaymath}
\left( h(x) + \frac{1}{2\log x}  - \frac{\log \log x - 2.25}{2\log^2 x} +  \frac{(\log \log x)^2 - 4.5 \log \log x + 22.51/3}{2 \log^3 x}\right)' = - 
\frac{f(\log x)}{g(\log x)r(x)\log^4 x} < 0
\end{displaymath}
for every $x \geq \exp(\exp(2.4))$. In addition, we have
\begin{displaymath}
\lim_{x \to \infty} \left( h(x) + \frac{1}{2\log x} - \frac{\log \log x - 2.25}{2\log^2 x} + \frac{(\log \log x)^2 - 4.5 \log \log x + 22.51/3}{2 \log^3 x} 
\right) = 0.
\end{displaymath}
So, we get
\begin{displaymath}
h(x) > - \frac{1}{2\log x} + \frac{\log \log x - 2.25}{2\log^2 x} -  \frac{(\log \log x)^2 - 4.5 \log \log x + 22.51/3}{2 \log^3 x}
\end{displaymath}
for every $x \geq \exp(\exp(2.4))$. Together with \eqref{5.3} and \eqref{5.4}, it follows that the desired inequality holds
% \begin{displaymath}
% \log \left( 1 + \frac{2R(n)}{p_n} \right) > - \frac{1}{2 \log n} + \frac{\log \log n - 2.25}{2\log^2n}  -  \frac{(\log \log n)^2 - 4.5 \log \log n + 
% 22.51/3}{2 \log^3 n}
% \end{displaymath}
for every positive integer $n \geq \exp(\exp(2.4))$. The remaining cases are checked with a computer.
\end{proof}

Next, we give an upper bound for $\log(1 + 2R(n)/p_n)$.

\begin{prop} \label{prop502}
For every positive integer $n \geq 6\,077$, we have
\begin{displaymath}
\log \left( 1 + \frac{2R(n)}{p_n} \right) < - \frac{1}{2 \log p_n} - \frac{1}{\log^2 p_n} - \frac{2.9}{2\log^2n \log p_n}.
\end{displaymath}
\end{prop}

\begin{proof}
First, we consider the case where $n \geq 78\,150\,372 \geq \exp(\exp(2.9))$. We define
\begin{displaymath}
s_2(x) = - \frac{x}{4} - \frac{x}{4\log x} + \frac{x(\log \log x - 2.9)}{4 \log^2 x}.
\end{displaymath}
By \cite[Theorem 1.7]{axler2017} and the definition of $R(n)$, we obtain $R(n) < s_2(n) < 0$. Hence,
\begin{displaymath}
\log \left( 1 + \frac{2R(n)}{p_n} \right) < \log \left( 1 + \frac{2s_2(n)}{p_n} \right).
\end{displaymath}
Since $2s_2(n)/p_n > -1$, we apply the inequality $\log(1+x) \leq x$, which holds for every $x > -1$, to get that
\begin{displaymath}
\log \left( 1 + \frac{2R(n)}{p_n} \right) < -\frac{n}{2p_n} - \frac{n}{2p_n \log n} + \frac{n(\log \log n - 2.9)}{2 p_n\log^2n }.
\end{displaymath}
Now, we use a lower bound for the prime counting function given by Dusart \cite[Th\'{e}or\`{e}me 1.10]{pd}, namely that $\pi(x) \geq x/\log x + 
x/\log^2x$ for every $x \geq 599$, with $x = p_n$ to obtain that
\begin{equation}
\log \left( 1 + \frac{2R(n)}{p_n} \right)  < - \frac{1}{2\log p_n} - \frac{1}{2\log^2 p_n} - \frac{1}{2\log n \log p_n} - \frac{1}{2\log n \log^2 p_n} + 
\frac{n(\log \log n - 2.9)}{2 p_n\log^2n}. \tag{5.5} \label{5.5}
\end{equation}
From Dusart \cite[Th\'{e}or\`{e}me 1.10]{pd} follows that $\pi(x) \leq x/\log x + 2x/\log^2x$ for every $x > 1$. Together with \eqref{5.5} and $\log \log n 
\geq 2.9$, we get that
\begin{align*}
\log \left( 1 + \frac{2R(n)}{p_n} \right) & < - \frac{1}{2\log p_n} - \frac{1}{2\log^2 p_n} - \frac{1}{2\log n \log p_n} - \frac{1}{2\log n \log^2 p_n} \\ 
& \p{\q\q} + \frac{\log \log n - 2.9}{2 \log^2n \log p_n} + \frac{\log \log n - 2.9}{\log^2n \log^2 p_n}. \nonumber
\end{align*}
Now we use \eqref{3.17} to get
% , we have
% \begin{displaymath}
% -\frac{1}{\log n} \leq - \frac{1}{\log p_n} - \frac{\log \log n}{\log^2 n} + \frac{(\log \log n)^2 - \log \log n + 1}{\log^2 n \log p_n} + \frac{P_8(\log 
% \log n)}{2 \log^3n\log p_n} - \frac{P_9(\log \log n)}{2 \log^4n\log p_n}
% \end{displaymath}
% for every $n \geq 2$, where $P_8(x) = 3x^2 - 6x + 5.2$ and $P_9(x) =  x^3 - 6x^2 + 11.4x - 4.2$. Since $P_9(x) > 0$ for every $x \geq 0.5$, we get
% \begin{displaymath}
% - \frac{1}{\log n} \leq - \frac{1}{\log p_n} - \frac{\log \log n}{\log^2 n} + \frac{(\log \log n)^2 - \log \log n + 1}{\log^2 n \log p_n} + \frac{P_8(\log 
% \log n)}{2 \log^3n\log p_n}.
% \end{displaymath}
% Applying this inequality to \eqref{6.6}, we get
\begin{align*}
\log \left( 1 + \frac{2R(n)}{p_n} \right) & < - \frac{1}{2\log p_n} - \frac{1}{\log^2 p_n} + \frac{(\log \log n)^2 + \log \log n - 4.8}{2 \log^2 n \log^2 p_n} 
+ \frac{P_8(\log \log n)}{4 \log ^3n\log^2 p_n} \\
& \p{\q\q} - \frac{1}{2\log n \log^2 p_n} - \frac{2.9}{2\log^2 n \log p},
\end{align*}
%Since
%\begin{displaymath}
%\frac{1}{2 \log m} + \frac{3.2}{2\log^2 m} \geq \frac{(\log \log m)^2 + \log \log m}{2\log^2 m} + \frac{P_8(\log \log m)}{4\log^3m}
%\end{displaymath}
%for every $m \geq 3$, we obtain
which implies the required inequality for every positive integer $n \geq 78\,150\,372$. For smaller values for $n$, we use a computer.
\end{proof}

In the direction of Proposition \ref{prop207}, we find the following upper bound for $\log(1 + 2R(n)/p_n)$ in terms of $n$, which leads to an improvement of 
the right-hand side inequality of \eqref{5.1}. 

\begin{kor} \label{kor503}
For every positive integer $n \geq 92$, we have
\begin{displaymath}
\log \left( 1 + \frac{2R(n)}{p_n} \right) < - \frac{1}{2 \log n} + \frac{\log \log n - 2}{2\log^2 n} + \frac{4\log \log n - 2.9}{\log^3 n} + 
\frac{2.9\log\log n}{2\log^4 n}.
\end{displaymath}
\end{kor}

\begin{proof}
First we consider the case where $n \geq 6\,077$. 
%Proposition \ref{prop604} implies that
%\begin{equation}
%\log \left( 1 + \frac{2R(n)}{p_n} \right) < - \frac{1}{2 \log p_n} - \frac{1}{\log^2 p_n}. \tag{6.7} \label{6.7}
%\end{equation}
From \eqref{3.16} follows that
\begin{equation}
- \frac{1}{\log p_n} \leq - \frac{1}{\log n} + \frac{\log \log n}{\log^2 n}. \tag{5.6} \label{5.6}
\end{equation}
Applying this to Proposition \ref{prop502}, we get that the inequality
\begin{displaymath}
\log \left( 1 + \frac{2R(n)}{p_n} \right) < - \frac{1}{2 \log n} + \frac{\log \log n}{2\log^2 n} - \frac{1}{\log n \log p_n} + \frac{\log \log n}{\log^2 n 
\log p_n} - \frac{2.9}{2\log^3n} + \frac{2.9\log\log n}{2\log^4n}.
\end{displaymath}
Again we use \eqref{5.6} to obtain that the required inequality holds for every positive integer $n \geq 6\,077$. We conclude by direct computation.
% \begin{displaymath}
% \log \left( 1 + \frac{2R(n)}{p_n} \right) < - \frac{1}{2 \log n} + \frac{\log \log n}{2\log^2 n} - \frac{1}{\log^2 n} + \frac{\log \log n}{\log^3 n} + 
% \frac{\log \log n}{\log^2 n \log p_n},
% \end{displaymath}
% which concludes the proof for every $n \geq ???$. For every $259 \leq n \leq ???$, we check the the required inequality with a computer.
\end{proof}

Compared with Proposition \ref{prop207} we establish the following more precise result.

\begin{kor} \label{kor504}
We have
\begin{displaymath}
\log \left( 1 + \frac{2R(n)}{p_n} \right) = - \frac{1}{2 \log n}  + \frac{\log \log n}{2 \log^2 n} + O \left( \frac{1}{\log^2 n} \right).
\end{displaymath}
\end{kor}

\begin{proof}
The claim follows directly from Proposition \ref{prop501} and Corollary \ref{kor503}.
\end{proof}

%%%%%%%%%%%%%%%%%%%%%%%%%%%%%%%%%%%%%%%%%%%%%%%%%

\section{New bounds for the ratio of $A_n$ and $G_n$}

% First, we recall the asymptotic formula for the ratio of $A_n$ and $G_n$ given in Corollary \ref{kor207}, namely
% \begin{displaymath}
% \frac{A_n}{G_n} = \frac{e}{2} + \frac{e}{4\log p_n} + \frac{e}{\log^2 p_n} + O \left( \frac{1}{\log^3p_n} \right). \tag{7.1} \label{7.1}
% \end{displaymath}
% and the identity \eqref{1.1}; i.e.
% \begin{displaymath}
% \log \frac{A_n}{G_n} = D(n) + \log \left( 1 + \frac{2R(n)}{p_n} \right) - \log 2.
% \end{displaymath}
Now we use ths identity \eqref{1.1} together with the explicit estimates for the quantities $D(n)$ and $\log(1+2R(n)/p_n)$ obtained in Section 3 and Section 5 
to derive upper and lower bounds for the ratio of $A_n$ and $G_n$ in the direction of \eqref{1.2}. We start with the following result.
 
\begin{thm} \label{thm601}
For every positive integer $n \geq 139$, we have
\begin{equation}
\frac{A_n}{G_n} > \frac{e}{2} + \frac{e}{4\log n} - \frac{e(\log \log n - 2.8)}{4 \log^2n}. \tag{6.1} \label{6.1}
\end{equation}
\end{thm}

\begin{proof}
First, we consider the case where $n \geq 465\,944\,315$. By \eqref{1.1}, Proposition \ref{prop306} and Proposition \ref{prop501}, we 
get that the inequality
\begin{displaymath}
\frac{A_n}{G_n} > \frac{e}{2} \cdot \exp \left( \frac{1}{2\log n} - \frac{\log \log n - 2.75}{2\log^2 n} - \frac{(\log \log n)^2 - 4.5 \log \log n + 22.51/3}{2
\log^3n} \right)
\end{displaymath}
holds. Notice that $0.15\log x > (\log \log x)^2 - 4.5 \log \log x + 22.51/3$ for every $x \geq 465\,944\,315$. Hence,
\begin{displaymath}
\frac{A_n}{G_n} > \frac{e}{2} \cdot \exp \left( \frac{1}{2\log n} - \frac{\log \log n - 2.6}{2\log^2 n} \right).
\end{displaymath}
Now we use the inequality $e^x \geq 1+x + x^2/2$, which holds for every nonnegative $x$, to get
\begin{displaymath}
\frac{A_n}{G_n} > \frac{e}{2} \cdot \left( 1 + \frac{1}{2\log n} - \frac{\log \log n - 2.85}{2\log^2 n}- \frac{\log \log n}{4 \log^3 n} + \frac{0.65}{\log^3
n}\right).
\end{displaymath}
Since the function $t \mapsto 2 \cdot 0.05 - (\log\log t - 4 \cdot 0.65)/(4\log t)$ is positive for every $t \geq 2$, the required lower bound for the ratio of 
$A_n$ and $G_n$ holds. A direct computation for every positive integer $n$ such that $139 \leq n \leq 465\,944\,314$ completes the proof.
\end{proof}

\begin{rema}
Hassani \cite{hassani2013} conjectured that there exists a real number $\alpha$ with $0 < \alpha < 9.514$ and a positive integer $n_0$ such that $A_n/G_n > e/2 
+ \alpha/\log n$ for every positive integer $n \geq n_0$. Theorem \ref{thm601} proves this conjecture.
\end{rema}

Now, we derive the inequality stated in Theorem \ref{thm102}. Here, we use \eqref{3.17} and Theorem \ref{thm601}.

\begin{kor} \label{kor602}
For every positive integer $n \geq 62$, we have
\begin{displaymath}
\frac{A_n}{G_n} > \frac{e}{2} + \frac{e}{4 \log p_n} + \frac{0.61e}{\log^2 p_n}.
\end{displaymath}
\end{kor}

\begin{proof}
First, let $n$ be a positive integer with $n \geq 1\,499\,820\,545$. Using \eqref{6.1} and \eqref{3.17}, we get
\begin{displaymath}
\frac{A_n}{G_n} > \frac{e}{2} + \frac{e}{4 \log p_n} + \frac{e}{4} \left( \frac{2.8}{\log^2 n} - \frac{(\log \log n)^2 - \log \log n + 1}{\log^3 n} - 
\frac{P_8(\log \log n)}{2 \log^4 n}\right),
\end{displaymath}
where $P_8(x) = 3x^2 - 6x + 5.2$. Now we apply \eqref{3.19} to get that the required inequality is valid.
%\begin{displaymath}
%\frac{A_n}{G_n} > \frac{e}{2} + \frac{e}{4 \log p_n} + \frac{2e}{5 \log^2 n}
%\end{displaymath}
%holds, which implies that required inequality
% holds for every positive integers $n \geq ???$.
We complete the proof by verifying the remaining cases with a computer.
\end{proof}

The following corollary confirms that the ratio of the arithmetic and geometric means of the prime numbers is always greater than $e/2$, as conjectured by 
Hassani \cite{hassani2013}.

\begin{kor} \label{kor603}
For every positive integer $n$, we have
\begin{displaymath}
\frac{A_n}{G_n} > \frac{e}{2}.
\end{displaymath}
\end{kor}

\begin{proof}
Corollary \ref{kor602} implies the validity of the required inequality for every positive integer $n \geq 62$. We verify the remaining cases with a computer.
\end{proof}

Next, we use Proposition \ref{prop305} and Proposition \ref{prop502} to find the following upper bound for the ratio of $A_n$ and $G_n$, which is stated in 
Theorem \ref{thm103}.

\begin{thm} \label{thm604}
For every positive integer $n \geq 294\,635$, we have
\begin{displaymath}
\frac{A_n}{G_n} < \frac{e}{2} + \frac{e}{4\log p_n} + \frac{1.52e}{\log^2 p_n}.
\end{displaymath}
\end{thm}

\begin{proof}
First, let $n$ be a positive integer with $n \geq 74\,004\,585$.
% \max\{ 74\,004\,585, 6\,077, \pi(e^{19.937})+1 = 24\,131\,125 \}.
By \eqref{1.1}, Proposition \ref{prop305}, Proposition \ref{prop502} and \eqref{4.3}, we obtain that
% \begin{displaymath}
% \frac{A_n}{G_n} < \frac{e}{2} \cdot \exp \left( \frac{1}{2\log p_n} + \frac{2.84}{\log^2 p_n} - \frac{2.9}{2\log^3 p_n} \right).
% \end{displaymath}
% Now we use \eqref{4.3} to get that
\begin{displaymath}
\frac{A_n}{G_n} < \frac{e}{2} \cdot \left( 1 + \frac{1}{2\log p_n} + \frac{9.02}{3\log^2 p_n} + \frac{1.33}{3\log^3 p_n} + \frac{13.2312}{3\log^4 p_n}\right).
\end{displaymath}
Since $\log p_n \geq 19.937$, we have $9.02/3 + 1.33/(3\log p_n) + 13.2312/(3\log^2 p_n) < 3.04$, which completes the proof for every positive integer $n \geq 
74\,004\,585$. We conclude by direct computation.
\end{proof}

Finally we use Theorem \ref{thm604} and the inequality \eqref{3.17} to prove the following result.

\begin{kor} \label{kor605}
For every positive integer $n \geq 2$, we have
\begin{displaymath}
\frac{A_n}{G_n} < \frac{e}{2} + \frac{e}{4\log n} - \frac{e( \log \log n - 6.44)}{4 \log^2n}.
\end{displaymath}
\end{kor}

\begin{proof}
By Theorem \ref{thm604} and the inequality \eqref{3.17}, we obtain that
\begin{equation}
\frac{A_n}{G_n} < \frac{e}{2} + \frac{e}{4\log n} - \frac{e( \log \log n - 6.08)}{4 \log^2n} + \frac{e((\log \log n)^2 - \log \log n + 1)}{4\log^3 n} + 
\frac{eP_8(\log \log n)}{8 \log^4 n} \tag{6.2} \label{6.2}
\end{equation}
for every positive integer $n \geq 294\,635$. Applying \eqref{3.19} to \eqref{6.2}, the claim follows for every positive integer $n \geq 1\,499\,820\,545$. A 
computer check shows the correctness of the required inequality for every positive integer $n$ satisfying $2 \leq n \leq 1\,499\,820\,544$.
\end{proof}

%The inequalities \eqref{1.2} imply that $A_n/G_n = e/2 + O(1/\log n)$. In the following corollary, we give a more accurate asymptotic formula for the ratio of
%the arithmetic and geometric means of the prime number.

%\begin{kor} \label{k706}
%We have
%\begin{displaymath}
%\frac{A_n}{G_n} = \frac{e}{2} + \frac{e}{4\log n} - \frac{e \log \log n}{4 \log^2 n}+ O \left( \frac{1}{\log^2n} \right).
%\end{displaymath}
%\end{kor}

%\begin{proof}
%The claim follows directly from Theorem \ref{thm701} and Corollary \ref{kor705}.
%\end{proof}

\begin{rema}
One of the conjectures concerning the ratio of $A_n$ and $G_n$ stated by Hassani \cite{hassani2013} is still open, namely that the sequence $(A_n/G_n)_{n \in 
\N}$ is strictly decreasing for every positive integer $n \geq 226$.
\end{rema}

%%%%%%%%%%%%%%%%%%%%%%%%%%%%%%%%%%%%%%%%%%%%%%%%%

\section*{Acknowledgement}
%I would like to thank the anonymous reviewers for their careful reading of this paper.
I would like to thank Mehdi Hassani for drawing my attention to the present subject. I would also like to express my great appreciation to Marc Del\'{e}glise
for the computation of several special values of Chebyshev's $\vartheta$-function. 
%He also wishes to thank the anonymous referees for their careful reading of this paper and their helpful comments.
%Furthermore, he wishes to thank Allysa Lumley for pointing out a mistaken sign in Lemma 8.

\end{document}